\documentclass[11pt,reqno]{amsart}

\usepackage{amssymb}

\usepackage[usenames,dvipsnames]{color}

\newtheorem{theorem}{Theorem}[section]
\newtheorem{proposition}[theorem]{Proposition}

\newtheorem{lemma}[theorem]{Lemma}
\newtheorem{corollary}[theorem]{Corollary}

\theoremstyle{definition}

\newtheorem{assumption}[theorem]{Assumption}

\newtheorem{remark}[theorem]{Remark}

\numberwithin{equation}{section}

\newcommand{\A}{\mathcal{A}}
\newcommand{\Sf}{\mathfrak{S}}

\newcommand{\B}{\mathcal{B}}
\newcommand{\Q}{\mathcal{Q}}
\newcommand{\eps}{\varepsilon}

\newcommand{\pd}{\partial}

\newcommand{\N}{\mathbb{N}}

\newcommand{\R}{\mathbb{R}}

\begin{document}

\title[Eigenvalue asymptotic of Robin Laplace operators]
{Eigenvalue asymptotic of Robin Laplace operators on two-dimensional
domains with cusps}

\author {Hynek Kova\v{r}\'{\i}k}

\address { Dipartimento di Matematica, Politecnico di Torino,  Corso Duca degli Abruzzi, 24, 10129 Torino, ITALY }

\email {hynek.kovarik@polito.it}


\date {\today}

\begin {abstract}
We consider Robin Laplace operators on a class of two-dimensional
domains with cusps. Our main results include the formula for the
asymptotic distribution of the eigenvalues of such operators. In
particular, we show how the eigenvalue asymptotic depends on the
geometry of the cusp and on the boundary conditions.
\end{abstract}

\maketitle




\section{Introduction}
Let $\Omega\subset\R^2$ be an open domain such that the spectrum of
the Dirichlet Laplacian $-\Delta_\Omega^D$ on $\Omega$ is discrete.
Denote by $N_\lambda(-\Delta_\Omega^D)$ the counting function of
$-\Delta_\Omega^D$, i.e. the number of eigenvalues of
$-\Delta_\Omega^D$ less than $\lambda$. The classical result by
H.~Weyl, \cite{we}, states that if $\Omega$ is bounded, then
\begin{equation} \label{weyl-cl}
N_\lambda(-\Delta_\Omega^D) \, = \, \frac{ \lambda}{4 \pi}\,
|\Omega| +o(\lambda) \qquad   \lambda\to\infty,
\end{equation}
where $|\Omega|$ denotes the volume of $\Omega$. The proof of
\eqref{weyl-cl} for unbounded domains with finite volume is due to M.~Birman, M.~Solomyak and
B.~Boyarski, see  e.g.~\cite{bs}. The situation is different for the Neumann Laplacian
$-\Delta_\Omega^N$. In this case equation \eqref{weyl-cl}, with
$N_\lambda(-\Delta_\Omega^N)$ in place of
$N_\lambda(-\Delta_\Omega^D)$, holds whenever $\Omega$ is bounded
and has sufficiently regular boundary, see e.g.~\cite{iv1,net,ns}
for the estimates on the rest term in \eqref{weyl-cl}. However, the
Neumann Laplacian might not satisfy \eqref{weyl-cl} (its spectrum
might even not be discrete) if $\Omega$ has rough boundary or if
$\Omega$ is unbounded, \cite{ber,ds,hss,jms,ns,sol1}.

Here we will focus on unbounded domains with regular boundary and we
will consider two-dimensional domains of the form
\begin{equation} \label{region-def}
\Omega = \{(x,y)\in\R^2\, : \, x>1,\, |y|<f(x)\}\, ,
\end{equation}
where $f:(1,\infty)\to\R$ is a positive function such that $f(x)\to
0$ as $x\to\infty$. Then the counting function
$N_\lambda(-\Delta_\Omega^D)$ of the Dirichlet Laplacian satisfies
\eqref{weyl-cl} as long as $f$ is integrable. If $f$ decays too
slowly, so that $|\Omega|=\infty$, then the spectrum of
$-\Delta_\Omega^D$ is still discrete, but
$N_\lambda(-\Delta_\Omega^D)$ grows super-linearly in $\lambda$, see
\cite{be,da,ro,si}. On the other hand, the spectrum of the Neumann
Laplacian $-\Delta_\Omega^N$ is discrete {\it if and only if}
\begin{equation} \label{compact-N}
\lim_{x\to\infty}\, \left( \int_1^x \frac{dt}{f(t)}\right) \left(
\int_x^\infty f(t)\, dt\right)=0.
\end{equation}
This remarkable fact was proved in \cite{eh}, see also \cite{mz}.
Asymptotic behaviour of $N_\lambda(-\Delta_\Omega^N)$ on domains of
this type was studied in \cite{be,ivrii,jms,net,sol1}. We would like
to point out that $f$ must decay faster than any power function for
\eqref{compact-N} to hold. We thus notice a huge difference between
the spectral properties of $-\Delta_\Omega^D$ and $-\Delta_\Omega^N$
on such domains.

Motivated by this discrepancy, we want to study the gap between
Dirichlet and Neumann Laplacians. To do so we consider a family of
Laplace operators on $\Omega$ which formally correspond
to the so-called Robin boundary conditions
\begin{equation} \label{robin-bc}
\frac{\pd u}{\pd n}\, (x,y) +h(x)\, u(x,y)=0, \qquad x>1, \, \, y=
\pm f(x),
\end{equation}
where $\frac{\pd u}{\pd n}$ denotes the normal derivative of $u$ and
$h:(1,\infty)\to\R_+$ is a sufficiently smooth bounded function. The
extreme cases $h\equiv 0$ and $h\equiv \infty$ correspond to Neumann
and Dirichlet Laplacians respectively. First question that arises is
under what conditions on $h$ and $f$ is the spectrum of the
associated Robin Laplacian discrete. Next we would like to know how
the coefficient $h(x)$ of the boundary conditions affects the
asymptotic distribution of eigenvalues of the Robin Laplacian.

The paper is organised as follows. In section \ref{asympt:main} we
formulate our main results, see Theorems \ref{principle} and
\ref{main-thm}. Similarly as in \cite{ber,jms}, we show that the
leading term of the eigenvalue asymptotic has two contributions, one
of which results from an auxiliary one-dimensional Schr\"odinger
operator. The boundary conditions affect the eigenvalue asymptotic
through the term $h(x)\sqrt{1+f'(x)^2}/f(x)$ which enters into the
potential of this operator, see equations \eqref {potentials} and
\eqref{H}. For some particular choices of $h$ and $f$ this
contribution can be calculated explicitly, the corresponding results
are given in section \ref{examples}.

The proofs of the main results are given in section \ref{proof}. Our
strategy is to treat separately the contribution to the counting
function from a finite part of $\Omega$ and from the tail. In
section \ref{step1} it is shown that the contribution from the
finite part satisfies the Weyl law \eqref{weyl-cl}. The key point of
the proof is to transform, in the remaining part of $\Omega$, the
problem to a Neumann Laplacian plus a positive potential that
reflects the boundary term, see section \ref{step 2}. To this end we
employ the technique known as ground state representation, which has
been recently used e.g. in \cite{fsw} to derive eigenvalue estimates
for Schr\"odinger operators with regular ground states, see also
\cite{fls}. Once this transformation is done, we show, by rather
standard arguments, that one part of the eigenvalue distribution of
such Neumann Laplacian with additional potential is asymptotically
(i.e. for $\lambda\to\infty$) equivalent to eigenvalue distribution
of a direct sum of certain one-dimensional Schr\"odinger operators,
see section \ref{step 3}. This enables us to prove Theorem
\ref{principle}. Finally, in the closing section \ref{gen} we
discuss some generalisations for Robin Laplacians with non symmetric
boundary conditions.


\section{Preliminaries and notation}
\label{prelim}

Given a self-adjoint operator $T$ with a purely discrete spectrum we
denote by $N_\lambda(T)$ the number of its eigenvalues, counted with
multiplicities, less than $\lambda$.  We will write $A \, \simeq \,
B$ if the operators $A$ and $B$ are unitarily equivalent and we will
use the notation
$$
f_1( \lambda) \sim f_2( \lambda) \quad \lambda\to\infty
\quad \Longleftrightarrow \quad
\lim_{ \lambda\to\infty} \frac{f_1( \lambda)}{f_2( \lambda) }=1.
$$
We will consider the eigenvalue behaviour of the Robin boundary value
problem in a weak sense. Therefore the main object of our interest
is the self-adjoint operator $A_\sigma$ in $L^2(\Omega)$ associated
with the closure of the quadratic form
\begin{equation} \label{q-form}
Q_\sigma[u]= \int_\Omega |\nabla u|^2\, dxdy + \int_1^\infty\!
\sigma(x)\left(|u(x,f(x))|^2+|u(x,-f(x))|^2\right)\, dx
\end{equation}
on $C^2_{0}(\bar\Omega)$. Here $C^2_{0}(\bar\Omega)$ denotes the
restriction to $\Omega$ of functions from $C^2(\R^2)$ such that for
each $y$ the support of $u(\cdot, y)$ is a compact subset of
$(1,\infty)$. The operator $A_\sigma$ formally corresponds to the
Laplace operator on $\Omega$ with Dirichlet boundary condition at
$\{x=1\}$ and mixed boundary conditions \eqref{robin-bc} at the rest
of the boundary, if we chose $\sigma$ such that
\begin{equation*}
\sigma(x) = h(x) \, \sqrt{1+f'(x)^2}\, .
\end{equation*}
\begin{remark}
Since we work under the assumption that $f'(x)\to 0$ as $x\to
\infty$, see below, and since the asymptotic of
$N_\lambda(A_\sigma)$ depends only on the behaviour of $\sigma$ at
infinity, from now on we will work with the function $\sigma$
instead of $h$.
\end{remark}

\noindent We will also need the following auxiliary potentials:
\begin{equation} \label{potentials}
V(x) = \frac 14\, \left(\frac{f'}{f}\right)^2+\frac 12\,
\left(\frac{f'}{f}\right)', \qquad W_\sigma(x)= V(x) +
\frac{\sigma(x)}{f(x)}.
\end{equation}

\noindent Throughout the whole paper we will suppose that $f$ satisfies

\begin{assumption} \label{ass-f1}
$f\in C^\infty(1,\infty)$ is positive and such that $f'(x)\leq 0$ for all $x$ large enough. Moreover,
\begin{equation} \label{f}
\lim_{x\to\infty}\, f(x) = \lim_{x\to\infty}\, f''(x) = 0.
\end{equation}
\end{assumption}

\noindent Note that \eqref{f} implies $f'(x)\to 0$ as $x\to\infty$, see Lemma \ref{landau} below.

\begin{lemma} \label{landau}
Let $f\in C^2(1,\infty)$ be a nonnegative function. Assume that $f$ and $|f''|$ are
bounded on $(1,\infty)$. For a given $x>1$ define $M_x= \sup_{s\geq x} f(s)$ and $M''_x=
\sup_{s\geq x} |f''(s)|$. Then
\begin{equation} \label{taylor-new}
(f'(x))^2 \leq 2\, M_x \, M''_x .
\end{equation}
\end{lemma}

\begin{proof}
Let $s>x$. The Taylor expansion of $f$ at the point $x$ gives
\begin{equation} \label{tayl-2}
f(s)- f(x) = t f'(x) +\frac{t^2}{2}\, f''(y), \qquad y\in [x,s], \quad t= s-x.
\end{equation}
On the other hand, $f\geq 0$ ensures that $|f(x)-f(s)| \leq M_x$ for all
$s>x$. This together with \eqref{tayl-2} implies that the inequality
$$
|f'(x)| \leq \, \frac{M_x}{t}\, + \frac{t\, M_x''}{2}
$$
holds for all $t>0$. Optimization with respect to $t$ then gives the result.
\end{proof}

\begin{remark}
Note that if we leave out the assumption $f\geq 0$, then the above proof still works with the modification
that now $|f(x)-f(s)| \leq 2 M_x$. This results into the Landau inequality, i.e.~inequality \eqref{taylor-new}
with the factor $2$ on the right hand side replaced by $4$.
\end{remark}

\noindent
The hypothesis on $\sigma$ are the following:

\begin{assumption} \label{ass-h}
The function $\sigma\in C^2(1,\infty)$ is non negative. Moreover,
$\sigma, \, \sigma'$ and $\sigma''$ are bounded and
\begin{equation}  \label{compact-h}
\lim_{x\to\infty}\, W_\sigma(x)\, =\infty.
\end{equation}
\end{assumption}

\noindent In order to formulate our next assumption, we introduce the operator
\begin{equation} \label{H}
\mathcal{H}_\sigma= -\frac{d^2}{dx^2}\, + W_\sigma(x)
\qquad \text{in\, \, } L^2(1,\infty)
\end{equation}
with Dirichlet boundary condition at $x=1$. More precisely, $\mathcal{H}_\sigma$ is the
operator generated by the closure of the quadratic form
$$
\int_1^\infty\left(|\psi'|^2+W_\sigma\, \psi^2 \right)\,  dx,
\qquad \psi \in C_0^2(1,\infty).
$$
Alongside with $\mathcal{H}_\sigma$ we will also consider the auxiliary operator
\begin{equation}
\B = -\partial_x^2 -\frac{1}{f^2(x)}\, \, \partial_y^2 \qquad
\text{in\, \, } L^2((1,\infty)\times(-1,1))
\end{equation} with Dirichlet boundary conditions.

\begin{assumption} \label{number}
For $0< \eps <1 $ we have
\begin{align}
N_\lambda((1\pm\eps)\, \mathcal{H}_\sigma) & =
N_\lambda(\mathcal{H}_\sigma)(1+\mathcal{O}(\eps)), \label{eps} \\
N_\lambda((1\pm\eps)\, \B) & = N_\lambda(\B)(1+\mathcal{O}(\eps))
\label{eps-B}
\end{align}
\end{assumption}

\begin{remark}
A similar assumption was made in \cite{jms}. Although this
assumption is essential for the approach used in the proof of
Theorem \ref{principle} below, it is natural to believe that the
statement holds under more general conditions. Note also that for
domains with finite volume \eqref{eps-B} holds automatically.
\end{remark}


\section{Main results}

\label{asympt:main}

\begin{theorem} \label{disc}
If \ref{ass-f1} and \ref{ass-h} are satisfied, then the spectrum of
$A_\sigma$ is discrete.
\end{theorem}

\begin{remark}
Contrary to the case of Neumann Laplacian, the spectrum of
$A_\sigma$ can be discrete also if the volume of $\Omega$ is
infinite. For example if $\sigma$ is constant, then
\eqref{compact-h} is automatically satisfied in view of the fact
that $f(x) V(x)\to 0$ as $x\to\infty$, see equation \eqref{taylor}.
On the other hand, condition \eqref{compact-h} is, unlike
\eqref{compact-N}, only sufficient.
\end{remark}

\begin{theorem} \label{principle}
Suppose that assumptions \ref{ass-f1},\, \ref{ass-h} and
\ref{number} are satisfied. Then
\begin{equation} \label{principle-eq}
N_\lambda(A_\sigma) \, \sim \, N_\lambda(-\Delta_\Omega^D) +
N_\lambda(\mathcal{H}_\sigma) \qquad \lambda\to\infty.
\end{equation}
\end{theorem}

\begin{remark}
The second term in \eqref{principle-eq} is a contribution from the
eigenvalues of the operator $A_\sigma$ restricted to the space of
functions which depend only on $x$. This is analogous to the case of
Neumann Laplacian, \cite{ds,jms, sol1}. On the other hand, the
presence of the boundary term $\sigma(x)$ enables us to apply
\eqref{principle-eq} also in the situation in which the Neumann
Laplacian does not have purely discrete spectrum.
\end{remark}

\begin{remark} \label{general}
Theorem \ref{principle} allows a straightforward generalisation to
Robin Laplacians with different boundary conditions on the upper and
lower boundary of $\Omega$, say given through functions
$\sigma_1(x)$ and $\sigma_2(x)$. In that case we only have to
replace $\sigma(x)$ in \eqref{potentials} by
$(\sigma_1(x)+\sigma_2(x))/2$, see section \ref{non-sym} for
details.
\end{remark}

\noindent For domains with finite volume Theorem \ref{principle} and
the Weyl formula \eqref{weyl-cl} give

\begin{theorem} \label{main-thm}
Let $|\Omega| < \infty$ and suppose that assumptions \ref{ass-f1},\,
\ref{ass-h} and \eqref{eps} are satisfied. Then
\begin{equation} \label{main}
N_\lambda(A_\sigma) \, \sim \, \frac {\lambda}{4\pi}\, |\Omega| +
N_\lambda(\mathcal{H}_\sigma) \qquad  \lambda\to\infty.
\end{equation}
\end{theorem}

\begin{remark}
Note that if $\sigma\equiv 0$, then the condition $|\Omega|<\infty$
is necessary for the spectrum of $A_0=-\Delta_\Omega^N$ to be
discrete, see \eqref{compact-N}. Hence in that case there is no
difference between Theorems \ref{principle} and \ref{main-thm} and
the resulting eigenvalue asymptotic agrees with the one obtained in
\cite{ber, jms}.
\end{remark}

\begin{corollary} \label{robin-w}
Let $|\Omega|<\infty$ and let $\sigma(x) = \sigma$ be constant.
Assume that $f$ satisfies \ref{ass-f1}. Then
\begin{align} \label{robin-weyl}
\limsup_{x\to\infty}\, x^2 f(x)\, = 0 & \quad \Longrightarrow \quad
N_\lambda(A_\sigma)\, \sim\,  \frac{|\Omega|}{4\pi}\, \,  \lambda  & \lambda\to\infty  \\
 \label{robin-linear}
\lim_{x\to\infty} x^2 f(x)\, = a^2  & \quad \Longrightarrow \quad
N_\lambda(A_\sigma)\, \sim \left(\frac{|\Omega|}{4\pi}+
\frac{|a|}{4\sqrt{\sigma}}\, \right)  \lambda  & \lambda\to\infty.
\end{align}
\end{corollary}

\begin{remark}
\noindent Equation \eqref{robin-weyl} provides a sufficient
condition on the decay of $f$ for the Weyl's law in the case of
constant $\sigma$. Notice that the borderline decay behaviour is $f(x) \sim x^{-2}$
which is in contrast to $f(x) \sim x^{-1}$ in the case of Dirichlet Laplacian.
The reason behind this is that the principle eigenvalues
of Robin and Dirichlet Laplacians on an interval of the width $2f(x)$
scale in a different way as $f(x)\to 0$. Observe also that \eqref{robin-linear} turns
into \eqref{robin-weyl} when $\sigma\to\infty$, as expected.
\end{remark}
If the volume of $\Omega$ is infinite, then we confine ourselves to
situations when $f$ is a power function. The asymptotic distribution
of the Dirichlet-Laplacian on such region is known, see \cite{ro},
\cite{si}. These results together with Theorem \ref{principle} yield

\begin{corollary} \label{vol-infty}
Let $f(x)=x^{-\alpha},\, 0<\alpha\leq 1$. If $\mathcal{H}_\sigma$
satisfies \eqref{eps}, then as $ \lambda\to\infty$ we have
\begin{eqnarray*}
N_\lambda(A_\sigma)  & \sim &   \frac{1}{\pi}\,
\left(\frac{2}{\pi}\right)^{\frac{1}{\alpha}}\zeta\left(\frac{1}
{\alpha}\right)\, B\left(1+\frac{1}{2\alpha},\, \frac 12\right)\,
\lambda^{\frac
12+\frac{1}{2\alpha}} + N_\lambda(\mathcal{H}_\sigma) \quad \alpha <1, \\
N_\lambda(A_\sigma)  & \sim &  \frac{1}{\pi}\, \,  \lambda \log
\lambda + N_\lambda(\mathcal{H}_\sigma) \quad \, \, \alpha =1,
\end{eqnarray*}
where $\zeta(\cdot)$ and $B(\cdot\, , \cdot)$ denote the Riemann
zeta and the Euler beta function respectively.
\end{corollary}


\subsection{Examples}
\label{examples}

We give the asymptotic of $N_\lambda(A_\sigma)$ for some concrete
choices of $f$ and $\sigma$.

\subsubsection{$f(x) = x^{-\alpha},\, \alpha >1, \, \,  \sigma(x) = \sigma=
\mbox{const}$.} Here
$$
W_\sigma(x) = \left(\frac{\alpha^2}{4}\, + \frac{\alpha}{2}\right)
x^{-2} + \sigma x^\alpha
$$
is convex and increasing at infinity so that assumption \ref{number}
is satisfied, see \cite[Chap. 7]{ti}. Theorem \ref{main-thm} in
combination with Theorem \ref{tit}, see Section \ref{auxiliary},
gives
\begin{equation} \label{h-const}
N_\lambda(A_\sigma) \, \sim \,  \frac{|\Omega|}{4 \pi}\,  \lambda\,
+\, \frac{1}{\alpha\pi}\, \sigma^{-\frac{1}{\alpha}}\,
B\left(\frac{1}{\alpha}\, , \, \frac 32\right)\,  \lambda^{\frac
12+\frac{1}{\alpha}}, \qquad  \lambda\to\infty.
\end{equation}
Note that, in agreement with Corollary \ref{robin-w},
$N_\lambda(A_\sigma)$ obeys Weyl's law as long as $\alpha >2$ and
for $\alpha=2$ the order of $ \lambda$ is linear, but the
coefficient is different from the one in the Weyl asymptotic. When
$\alpha < 2$, then the behaviour of $N_\lambda(A_\sigma)$ for $
\lambda\to\infty$ is fully determined by the second term on the
right hand side of \eqref{h-const}.

\subsubsection{$f(x) = x^{-\alpha},\, 0<\alpha \leq 1, \, \,
\sigma(x) =  \sigma\, x^{-\beta}$.}
Assumptions \ref{ass-h} is satisfied if and only if $0 \leq \beta <
\alpha. $ For these values of $\beta$ Corollary \ref{vol-infty} and
Theorem \ref{tit} give
$$
N_\lambda(A_\sigma)\,  \sim\,
\frac{\sigma^{-\frac{1}{\alpha-\beta}}}{(\alpha-\beta)\pi}\, \,
B\left(\frac{1}{\alpha-\beta}\, , \, \frac 32\right)\,
\lambda^{\frac 12+\frac{1}{\alpha-\beta}}, \qquad  \lambda\to
\infty.
$$


\section{Auxiliary material}
\label{auxiliary}

\noindent In this section we collect some auxiliary material, which
will be used in the proof of the main results. First we fix some
necessary notation. Given a continuous function $q:(1,\infty)\to \R$
such that $q(x)\to \infty$ as $x\to\infty$, we denote by
$T^{D,D}_{(a,b)}$ the operator in $L^2(a,b)$ acting as
$$
T^{D,D}_{(a,b)} = -\frac{d^2}{dx^2} +q(x), \qquad 1\leq a < b <
\infty
$$
with Dirichlet boundary conditions at $x=a$ and $x=b$. Operators
$T^{D,N}_{(a,b)}, \, T^{N,N}_{(a,b)}$ and $T^{N,D}_{(a,b)}$ are
defined accordingly. For $b=\infty$ we use the simplified notation
$T^{D}_{(a,\infty)}$ etc. to indicate the corresponding boundary
condition at $x=a$. It is well known that imposing Dirichlet
boundary condition at $x=a$ is a rank one perturbation. Variational
principle thus implies that
\begin{equation} \label{rank1}
0\leq
N_\lambda\big(T^{N}_{(a,\infty)}\big)-N_\lambda\big(T^{D}_{(a,\infty)}\big)\leq
1\quad \forall\, a.
\end{equation}

\begin{lemma} \label{1-dim}
Suppose that $q(x)$ is a continuous function such that $q(x)\to
\infty$ as $x\to\infty$. Then for any $s >1$ it holds
\begin{equation} \label{equiv}
N_\lambda \big( T^{N}_{(1,\infty)} \big) \, \sim \,
N_\lambda\big( T^D_{(1,\infty)} \big)\, \sim\, N_\lambda\big(
T^{D}_{(s,\infty)} \big) \, \sim \, N_\lambda\big(
T^{N}_{(s,\infty)} \big) \qquad \lambda\to\infty.
\end{equation}
\end{lemma}

\begin{proof}
In view of \eqref{rank1} it suffices to consider the Dirichlet
operator only. Let $I_ \lambda:=\{x>s:\, q(x) <  \lambda/2\}$. Then
$$
N_\lambda\big(T^{D}_{(s,\infty)}\big) \, \geq\, N_{\frac{
\lambda}{2}}\big(-\frac{d^2}{dx^2}\big)^{Dir}_{L^2(I_ \lambda)}
\, \geq \frac{\sqrt{ \lambda}}{\pi\sqrt{2}}\, |I_ \lambda|\,
(1+o(1))\quad  \lambda\to\infty,
$$
where the superscript $Dir$ indicates Dirichlet boundary conditions
at the end points of $I_\lambda$. Since $|I_ \lambda|\to\infty$ as $
\lambda\to\infty$, this shows that $\liminf_{ \lambda\to\infty}\,  \lambda^{-1/2}\,
N_\lambda (T^{D}_{(s,\infty)} ) = \infty$. In view of the equation
$$
N_\lambda\big( T^{D,N}_{(1,s)}\big)\, \sim \, N_\lambda\big(
T^{D,D}_{(1,s)}\big) = \, \mathcal{O}(\sqrt{ \lambda}) \, \qquad
 \lambda\to\infty \quad \forall\, s>1,
$$
the result follows from the Dirichlet-Neumann bracketing (by putting additional boundary conditions at $x=s$), see e.g. \cite[Chap.13]{rs}.
\end{proof}

\noindent Under certain additional assumptions one can recover the
eigenvalue distribution of such operators from the potential $q$.
The following theorems are due to \cite[Chap. 7]{ti}:

\begin{theorem}[Titchmarsh] \label{tit}
Suppose that $q(x)$ is continuous increasing unbounded function,
that $q'(x)$ is continuous and $ x^3 q'(x) \to \infty$ as
$x\to\infty$. Then
\begin{equation} \label{cl}
N_\lambda\left(T^{D}_{(s,\infty)}\right) \, \sim \, \frac{1}{\pi}\,
\int_{s}^\infty \left( \lambda-q(x)\right)^{\frac 12}_+\, dx,\qquad
\lambda\to\infty.
\end{equation}
\end{theorem}

\begin{theorem}[Titchmarsh] \label{tit-2}
Suppose that $q(x)$ is continuous increasing and convex at infinity.
Then \eqref{cl} holds true.
\end{theorem}

\noindent A simple combination of the above results gives

\begin{lemma} \label{enough}
Assume that there exists some $x_c$ such that $q$ satisfies the
hypothesis of Theorem \ref{tit} or \ref{tit-2} for all $x>x_c$. Then
for any $s \geq x_c$ we have
\begin{equation} \label{egal}
N_\lambda\left(T^{D}_{(1,\infty)}\right) \, \sim \, \frac{1}{\pi}\,
\int_{1}^\infty \left( \lambda-q(x)\right)^{\frac 12}_+\, dx \, \sim
\, \frac{1}{\pi}\, \int_{s}^\infty \left( \lambda-q(x)\right)^{\frac
12}_+\, dx  \qquad \lambda\to\infty.
\end{equation}
\end{lemma}

\noindent Next we consider the operators
$$
\B_n^{N/D} = -\partial_x^2 -\frac{1}{f^2(x)}\, \, \partial_y^2 \quad
\text{in\, \, } L^2((n,\infty)\times(-1,1))
$$
subject to Dirichlet boundary conditions on $(n,\infty)\times
(\{1\}\cup \{-1\})$ and Neumann/Dirichlet boundary condition on
$\{n\}\times(-1,1)$ respectively. We have

\begin{lemma} \label{B-op}
For any $n\in\N$ it holds
\begin{align}
N_\lambda\left( \B_n^{N}\right) & \sim \, \, N_\lambda\left(\B_n^{D}\right) \,
\sim \, \frac{ \lambda}{2\pi}\, \, \int_n^\infty f(x)\, dx \qquad
 \lambda\to\infty \quad \text{if \, \,} |\Omega| < \infty  \label{B-2}, \\
 N_\lambda\left( \B_n^{N}\right) & \sim \, \, N_\lambda\left(\B_n^{D}\right)\,
\sim\, N_\lambda(\B) \qquad \qquad \qquad \,  \lambda\to\infty \quad \text{if\,
\, } |\Omega|=\infty. \label{B-1}
\end{align}
\end{lemma}

\begin{proof}
Equation \eqref{B-2} for $\B_n^D$ follows directly from \cite[Thm.
1.2.1]{sv}. Hence it remains to prove \eqref{B-2} for $\B_n^N$ and
\eqref{B-1}. Note that
$$
N_\lambda(\B_n^D) =  \sum_{k=1}^\infty  N_\lambda(L^D_{k,n})\, ,
\quad N_\lambda(\B_n^N) =  \sum_{k=1}^\infty  N_\lambda(L^N_{k,n}) ,
\quad \mbox{where} \quad  L^{N/D}_{k,n} = -\frac{d^2}{dx^2}
+\frac{\pi^2 k^2}{4 f(x)^2}\
$$
are one-dimensional operators acting in $L^2(n,\infty)$ with
Neumann/Dirichlet boundary conditions on $x=n$. Obviously there
exists a positive constant $c$ such that for any $n$ and any $k$ the
operator inequality $L^D_{k,n} \geq L^N_{k,n} \geq c\, k^2$ holds.
This means that there exists some $K( \lambda)$ with $K( \lambda) =
\mathcal{O}(\sqrt{ \lambda})$ as $ \lambda\to\infty$ and such that
$$
\sum_{k\geq 1} N_\lambda(L^N_{k,n}) = \sum_{k\geq 1}^{K( \lambda)}
N_\lambda(L^N_{k,n}), \quad \sum_{k\geq 1} N_\lambda(L^D_{k,n}) = \sum_{k\geq
1}^{K( \lambda)} N_\lambda(L^D_{k,n})
$$
Moreover, since $ 0 \leq N_\lambda(L^N_{k,n}) - N_\lambda(L^D_{k,n})
\leq 1$ holds for all $n\in\N$ and for all $k\geq 1$, see
\eqref{rank1},
\begin{equation} \label{sqrt}
\sum_{k\geq 1} N_\lambda(L^D_{k,n}) =  \sum_{k\geq 1}^{K( \lambda)}
N_\lambda(L^D_{k,n}) \, \leq \, \sum_{k\geq 1}^{K( \lambda)} N_\lambda(L^N_{k,n}) \,
\leq \, \sum_{k\geq 1} N_\lambda(L^D_{k,n}) + \mathcal{O}(\sqrt{ \lambda}).
\end{equation}
The latter implies \eqref{B-2} since $N_\lambda(\B_n^D)$ grows
linearly in $ \lambda$ when $|\Omega| <\infty$ as mentioned above.
To prove \eqref{B-1}  we consider the operators $\B_{n,m}^D$
obtained from $\B_n^D$ by putting additional Dirichlet boundary
condition at $\{x=m\},\, m>n$. From \cite[Thm. 1.2.1]{sv} we get
$$
\liminf_{ \lambda\to\infty}\,  \lambda^{-1} N_\lambda(\B_n^D) \geq
\liminf_{ \lambda\to\infty}\,  \lambda^{-1} N_\lambda(\B_{n,m}^D) = \frac{1}{2\pi}\,
\int_n^m\, f(x)\, dx\qquad \forall\, m>n,
$$
which implies, by letting $m\to\infty$, that $\liminf_{ \lambda\to\infty}\,  \lambda^{-1} N_\lambda(\B_n^D) = \infty$. In view of
\eqref{sqrt} we obtain $N_\lambda\left( \B_n^{N}\right) \sim
N_\lambda\left(\B_n^{D}\right)$. Finally, from the Dirichlet-Neumann
bracketing we deduce that $N_\lambda\left(
\B\right) \sim N_\lambda\left(\B_n^{D}\right)$.
\end{proof}


\section{Proofs of the main results}
\label{proof}

\noindent As mentioned in the introduction, the idea of the proof is
to split $N_\lambda(A_\sigma)$ into two parts corresponding to the
contribution from a finite part of $\Omega$ and from the tail.

\subsection{Step 1}
\label{step1} Here we show that the contribution from the part of
$\Omega$ where $x<n$ obeys the Weyl asymptotic irrespectively of the
boundary conditions. Let us define
\begin{align*}
\Omega_{n} & := \left\{(x,y) \in\Omega\, : \, 1<x< n \right\}, \quad
E_n := \Omega\setminus\Omega_n.
\end{align*}
We denote by $Q^N_{n,l}$ and $Q^N_{n,r}$ the quadratic forms defined
by the reduction of $Q_\sigma$ on $\Omega_{n}$ and $E_n$ and acting
on the functions from $C^2(\overline{\Omega}_{n})$ and
$C_0^2(\overline{E_n})$ respectively. Moreover, let $T^N_n$ and
$S^N_n$ be the operators associated with the closures of the forms
$Q^N_{n,l}$ and $Q^N_{n,r}$.

Similarly we denote by $Q^D_{n,l}$ and $Q^D_{n,r}$ the respective
quadratic forms which are defined in the same way as $Q^N_{n,l}$ and
$Q^N_{n,r}$ but with the additional Dirichlet boundary condition at
$\{x=n\}$. We then denote by $T^D_n$ and $S^D_n$ the operators
associated with the closures of the forms $Q^D_{n,l}$ and
$Q^D_{n,r}$. From the Dirichlet-Neumann bracketing we obtain
the operator inequality
\begin{equation} \label{two-sided-operators}
T^N_n \oplus S^N_n \leq A_\sigma \leq T^D_n \oplus S^D_n, \quad
n\in\N,
\end{equation}
which implies that
\begin{equation} \label{two-sided}
N_\lambda(T^D_n) + N_\lambda(S^D_n) \leq N_\lambda(A_\sigma) \leq
N_\lambda(T^N_n) + N_\lambda(S^N_n), \quad n\in\N,\, \, \lambda>0.
\end{equation}

\begin{lemma} \label{finite-part}
For any $n\in\N$ it holds
\begin{equation} \label{weyl-bis}
\lim_{ \lambda\to\infty} \, \lambda^{-1}\, N_\lambda(T^D_n) = \lim_{
\lambda\to\infty} \, \lambda^{-1}\, N_\lambda(T^N_n) \, =
\frac{1}{2\pi}\, \int_1^n f(x)\, dx.
\end{equation}
\end{lemma}

\begin{proof}
Fix $n\in\N$. Since $\sigma$ is bounded and $H^1(-f(x),f(x))$ is for
every $x\in (1,n)$ continuously embedded into
$L^\infty(-f(x),f(x))$, it follows that there exists a constant
$c_n$ such that
$$
\|\nabla u\|^2_{L^2(\Omega_n)}+\|u\|^2_{L^2(\Omega_n)}\, \leq\,
Q^N_{n,l}[u] +\|u\|^2_{L^2(\Omega_n)}\, \leq \, c_n \left(\,
\|\nabla u\|^2_{L^2(\Omega_n)} +\|u\|^2_{L^2(\Omega_n)}\right)
$$
holds for all $u\in C^2(\overline{\Omega_n})$. Hence the domain of
the closure of the quadratic form $Q^N_{n,l}$ is a subset of
$H^1(\Omega_n)$. The same reasoning shows that the domain of the
closure of $Q^D_{n,l}$ contains the space $H_0^1(\Omega_n)$. From
the fact that $\sigma \geq 0$ and from the variational principle we
thus conclude that
\begin{equation} \label{dn-bracket}
N_\lambda(-\Delta^D_{\Omega_n}) \leq N_\lambda(T^D_n) \leq
 N_\lambda(T^N_n) \leq N_\lambda(-\Delta^N_{\Omega_n}),
\end{equation}
where $-\Delta^D_{\Omega_n}$ and $-\Delta^N_{\Omega_n}$ denote the
Dirichlet and Neumann Laplacian on $\Omega_n$ respectively. Since
$\Omega_n$ has the $H^1-$extension property, the Weyl formula
$$
\lim_{ \lambda\to\infty} \, \lambda^{-1}\,
N_\lambda(-\Delta^D_{\Omega_n})= \lim_{ \lambda\to\infty} \,
\lambda^{-1}\, N_\lambda(-\Delta^N_{\Omega_n}) =
\frac{|\Omega_n|}{4\pi}
$$
holds for both $-\Delta^D_{\Omega_n}$ and $-\Delta^N_{\Omega_n}$, see \cite{me,bs},  \cite{ns}. In
view of \eqref{dn-bracket}, this completes the proof.
\end{proof}


\subsection{Step 2} \label{step 2}
Next we will treat the contribution to the counting function
$N_\lambda(A_\sigma)$ from the tail of $\Omega$. Our first aim is to
transform the boundary term in \eqref{q-form} into en  effective
additional potential. To this end we use a ground state
representation for the test functions $\psi$. Let $\mu(x)$ be the
first eigenvalue of the one-dimensional problem
\begin{align} \label{robin}
-\partial_y^2\, v(x,y) & = \mu(x)\, v(x,y), \\
\partial_y v(x,-f(x)) = \sigma(x)\, v(x,-f(x)),\ \  &
\ \ \partial_y v(x,f(x)) = -\sigma(x)\, v(x,f(x)) \nonumber
\end{align}
with the corresponding eigenfunction $v$. By lemma \ref{implicit}
$0<v\leq 1$ and $v(x,y)\to 1$ as $x\to\infty$ uniformly in $y$.
Moreover, $v\in C^2(\overline{E}_{n})$. Thus every function $\psi\in
D(\Q^N_n)$ can be written as
\begin{equation} \label{factor-N}
\psi(x,y) = v(x,y)\, \varphi(x,y), \quad \varphi\in
C_0^2(\overline{E}_{n})\, .
\end{equation}
Similarly, for every function $\psi\in D(\Q^D_n)$ we have
\begin{equation} \label{factor-D}
\psi(x,y) = v(x,y)\, \varphi(x,y), \quad \varphi\in
C_0^2(\overline{E}_{n})\cap \left\{\varphi :\, \varphi(n, \cdot) =0
\right\}
\end{equation}
In view of \eqref{factor-N} and \eqref{factor-D} we can thus
identify $Q^N_n[\psi]$ and $Q^D_n[\psi]$ with quadratic forms
$\Q^N_n[\varphi]$ and $\Q^D_n[\varphi]$ given by
$$
\Q^{N / D}_n[\varphi] = Q^{N /D}_n[v\, \varphi].
$$
and acting in the weighted space $L^2(E_{n}, v^2 dxdy)$. The forms
$\Q^N_n[\varphi]$ and $\Q^D_n[\varphi]$ are defined on $
D(\Q^N_n)=C_0^2(\overline{E}_{n})$ and $D(\Q^D_n) =
C_0^2(\overline{E}_{n})\cap \left\{\varphi :\, \varphi(n, \cdot) = 0
\right\}$ respectively. A straightforward calculation based on
integration by parts in $y$ then gives
\begin{align} \label{gr-state}
\Q^{N,D}_n[\varphi] & = \int_{E_{n}} \left( |\partial_x (v
\varphi)|^2 +
 \mu(x)\, v^2\, |\varphi|^2+ v^2\, |\partial_y\varphi|^2 \right )\, dxdy.
\end{align}

\noindent Since $v\to 1$ and $\mu(x)\sim \sigma(x)/f(x)$ as
$x\to\infty$, see appendix, it is natural to compare $\Q^{N,D}_n$
with the quadratic form
$$
q_n[\varphi]= \int_{E_n} \! \big( |\partial_x\varphi|^2
+|\partial_y\varphi|^2\, + \frac{\sigma(x)}{f(x)}\,
|\varphi|^2\big)\, dxdy .
$$
Let $\Sf_n^N$ and $\Sf_n^D$ be the operators in $L^2(E_n)$ generated
by the closures of the quadratic form $q_n[u]$ on $D(\Q^N_n)$ and
$D(\Q^D_n)$ respectively.

\begin{lemma} \label{lem-intermed}
Suppose that assumptions \ref{ass-f1} and \ref{ass-h} are satisfied.
For any $\eps$ there exists an $N_\eps$ such that for all $n>N_\eps$
\begin{equation} \label{intermed}
N_\lambda(S^N_n) \leq N_\lambda((1-\eps)\Sf_n^N -\eps), \quad
N_\lambda(S^D_n) \geq N_\lambda((1+\eps)\Sf_n^D +\eps) \quad
\lambda>0 .
\end{equation}
\end{lemma}

\begin{proof}
Let $\epsilon>0$ and let $\varphi$ belong to the domain of the
quadratic forms $\Q^N_{n}\, (\Q^D_{n}$). From the fact that
$$
\lim_{x\to\infty}\, v(x,y) =1,\quad \lim_{x\to\infty}\,
\partial_x v(x,y) =0 \, \, \, (\text{uniformly\, in\, }\, y), \quad
\lim_{x\to\infty}\, \frac{\mu(x) f(x)}{ \sigma(x)} = 1 \, ,
$$
see Lemma \ref{implicit}, and from the estimate
$$ | 2v\partial_x v \, \varphi\, \partial_x \varphi|
\leq \, \epsilon\, |\partial_x \varphi|^2v^2 +\epsilon^{-1}\,
|\varphi|^2 |\partial_x v|^2
$$
we conclude that for $n$ large enough
\begin{align}
\Q^N_{n}[\varphi] & \geq (1-\epsilon) \int_{E_n} \left(
|\partial_x\varphi|^2 +|\partial_y\varphi|^2\, +
\frac{\sigma(x)}{f(x)}\, |\varphi|^2\right)\, dxdy  -\epsilon\,
\|\varphi\|^2_{L^2(E_n)}\nonumber \\
\Q^D_{n}[\varphi] & \leq (1+\epsilon) \int_{E_n} \left(
|\partial_x\varphi|^2 +|\partial_y\varphi|^2\, +
\frac{\sigma(x)}{f(x)}\, |\varphi|^2\right)\, dxdy + \epsilon\,
\|\varphi\|^2_{L^2(E_n)}\, . \label{equiv-norm}
\end{align}
Moreover, by \eqref{eq-v} we also have $|v|\leq 1$ so that (still
for $n$ large enough)
$$
(1-\epsilon) \|\varphi\|^2_{L^2(E_n)} \leq  \int_{E_n}\, |\varphi|^2
v^2\, dxdy \, \leq \, \|\varphi\|^2_{L^2(E_n)}.
$$
Equation \eqref{intermed} then follows from the variational
principle by choosing $\epsilon$ in appropriate way (depending on
$\eps$).
\end{proof}


\subsection{Step 3}
\label{step 3}
We transform the problem of studying the Laplace
operator on $E_n$ to the problem of studying a modified operator on
the simpler domain
$$
D_n = (n,\infty)\times (-1,1) .
$$
To this end we introduce the transformation $U:L^2(E_n) \to
L^2(D_n)$ defined by
$$
(U \varphi)(x,t) = \sqrt{f(x)}\, \, \varphi(x,f(x)\, t),\quad
(x,t)\in D_n.
$$
Let $\A_n^{N}$  and $\A_n^{D}$ be the operators associated with the
closure of the form
\begin{equation} \label{new-form}
\widehat{Q}_n[u] := q_n[U^{-1} u],
\end{equation}
on $C_0^2(\overline{D}_{n})$ and $C_0^2(\overline{D}_{n})\cap
\left\{u :\, u(n, \cdot) = 0 \right\}$ respectively. Since $U$ maps
$L^2(E_n)$ unitarily onto $L^2(D_n)$ and $U\,
C_0^2(\overline{E}_{n})=C_0^2(\overline{D}_{n})$, the variational
principle gives
\begin{equation} \label{n-equal} N_\lambda(\A_n^N) =
N_\lambda(\Sf_n^N),\quad N_\lambda(\A_n^D) = N_\lambda(\Sf_n^D).
\end{equation}
By a direct calculation
\begin{equation*} \label{bordel}
\widehat{Q}_n[u] = \int_{D_n} \! \big( |\partial_x u|^2+W_\sigma\, u^2
-2t\, \frac{f'}{f}\,
\partial_x u\partial_t u + \frac{f'^2}{f^2}(
t\, u\partial_t u+t^2|\partial_t u|^2) + \frac{1}{f^2}\, |\partial_t
u|^2\big)\, dxdt .
\end{equation*}

\noindent Now Let $\eta \in (0,1)$ be arbitrary. Since $|t| \leq 1$
we get
\begin{align*}
\big |2t\, \frac{f'}{f}\, \partial_x u\, \partial_t u\big| & \,
\leq \, \eta \, |\pd_x u|^2 + \eta^{-1}\, \frac{f'^2}{f^2}\, |\pd_t
u|^2, \quad \frac{f'^2}{f^2}\, |t\, u\, \pd_t u|  \, \leq \,
\eta^{-1} \, \frac{f'^4}{f^2}\, \, |u|^2+ \frac{\eta}{f^2}\, \,
|\pd_t u|^2 .
\end{align*}
Moreover, from \eqref{taylor-new} and from the fact $f$ is decreasing at infinity, by assumption  \ref{ass-f1}, it follows that
\begin{equation} \label{taylor}
f'(x)^2 \leq 2 f(x)\, \sup_{s\geq x}\, |f''(s)|,
\end{equation}
for all $x$ large enough. Since $f''\to 0$ as $x\to\infty$, for
any $\eta\in (0,1)$ there clearly exists an $N_\eta$ such that for any
$n>N_\eta$ it holds
\begin{align} \label{pm-estim}
\widehat{Q}_n[u] & \lessgtr  \int_{D_n}\!  \big( (1\pm \eta)
|\partial_x u|^2+ W_\sigma\, u^2 + \frac{1\pm 2\eta}{f^2}\,
|\partial_t u|^2 \pm \eta u^2 \big )\, dxdt .
\end{align}

\noindent We denote by $H^N_n$ and $H^{D}_n$ the operators acting in
$L^2(\Omega_{n,r})$ associated with the closures of the quadratic
form
$$
\int_{D_n} \big( |\partial_x u|^2 +\frac{|\partial_t
u|^2}{f^2(x)}\, + W_\sigma(x)\, u^2\big)\, dxdt
$$
defined on $C_0^2(\overline{D}_{n})$ and
$C_0^2(\overline{D}_{n})\cap \left\{u :\, u(n, \cdot) = 0 \right\}$
respectively.

\begin{lemma} \label{lem6.3}
Suppose that assumptions \ref{ass-f1} and \ref{ass-h} are satisfied.
For any $\eps$ there exists an $N_\eps$ such that for all $n>N_\eps$
and any $ \lambda>0$ it holds
\begin{equation} \label{2-side-eps}
N_\lambda(\A^N_n) \leq N_\lambda((1-\eps) H^N_n), \quad
N_\lambda(\A^D_n) \geq N_\lambda((1+\eps) H^D_n).
\end{equation}
\end{lemma}

\begin{proof}
In view of the fact that $W_\sigma(x)\to\infty$ the statement
follows from \eqref{pm-estim}.
\end{proof}

\noindent Next we observe that since $W_\sigma$ depends only on $x$,
the matrix representations of the operators $H^N_n$ and $H^D_n$ in
the basis of (normalised) eigenfunctions of the operator
$-f(x)^{-2}\, \frac{d^2}{dt^2}$ on the interval $(-1, 1)$ with
Neumann boundary conditions are diagonal. We thus have the following
unitary equivalence:
\begin{equation} \label{ort-sum}
H^N_n \simeq \bigoplus_{k=0}^\infty \mathcal{H}^N_{k,n}\, , \quad
H^D_n \simeq \bigoplus_{k=0}^\infty \mathcal{H}^D_{k,n}, \qquad
\mathcal{H}^{N/D}_{k,n}= -\frac{d^2}{dx^2}\, + W_\sigma(x) +
\frac{k^2\pi^2}{4 f(x)^2}\, ,
\end{equation}
where $\mathcal{H}^{N/D}_{k,n}$ are operators in $L^2(n,\infty)$
with Neumann/Dirichlet boundary condition at $x=n$. We denote
$$
\mathcal{H}_{0,n}^N = \mathcal{H}^N_{n}, \quad \mathcal{H}_{0,n}^D =
\mathcal{H}^D_{n}\, .
$$

\noindent As a consequence of \eqref{ort-sum} we get

\begin{proof}[Proof of Theorem \ref{disc}]
We make use of inequality \eqref{two-sided-operators} for some fixed
$n$ and show that the operator on the left hand side of
\eqref{two-sided-operators} has purely discrete spectrum. By general
arguments of the spectral theory this will imply the statement.
Since the spectrum of $T^N_n$ is discrete, it suffices to show that
the same is true for $S^N_n$. In view of Lemma \ref{lem-intermed}
and equations \eqref{n-equal}, \eqref{2-side-eps} it is enough to
prove the discreteness of the spectrum of $H^{N}_n$. By
\eqref{ort-sum} we have
$$
\mbox{spect}(H^{N}_n) = \cup_{k=0}^\infty\,
\mbox{spect}(\mathcal{H}^N_{k,n}),
$$
First we notice that $\mbox{spect}(\mathcal{H}^N_{k,n})$ is purely
discrete for each $k$ and $n$. Indeed, a sufficient condition for
the spectrum of $\mathcal{H}^N_{k,n}$ to be purely discrete is that
\begin{equation} \label{infinity}
 W_\sigma(x) + \frac{k^2\pi^2}{4 f(x)^2} \to \infty
\quad \text{as\, \, } x\to\infty\, ,
\end{equation}
see e.g. \cite[Thm. 13.67]{rs}, which is a direct consequence of
assumption (\ref{ass-h}). Hence the spectrum of $H^N_n$ is pure
point, i.e. consists only of eigenvalues. Moreover, since $f^2(x)
W_\sigma(x) \to 0$ as $x\to\infty$ by \eqref{taylor} and boundedness
of $\sigma$, it is easy to see that
$$
\forall\, n \quad \inf\, \mbox{spect}(\mathcal{H}^N_{k,n}) \to
\infty\qquad \mbox{as}\quad k\to\infty.
$$
Hence all the eigenvalues in the spectrum of $H^N_n$ have finite
multiplicity and $\mbox{spect}(H^N_n)$ contains no finite point of
accumulation. This means that $\mbox{spect}(H^N_n)$ is discrete.
\end{proof}


\begin{proof}[Proof of Theorem \ref{principle}] Case
$|\Omega|<\infty$. By Lemma \ref{1-dim}  the asymptotic behaviour of
$N_\lambda(\mathcal{H}^{N,D}_n)$ does not depend on the boundary
condition at $x=n$, nor on $n$ itself:
\begin{equation} \label{asymp-equal}
N_\lambda(\mathcal{H}_\sigma) \sim  N_\lambda(\mathcal{H}^N_{n})
\sim N_\lambda(\mathcal{H}^D_{n}) \quad  \lambda\to\infty, \, \, \,
\forall\, n\in\N\, .
\end{equation}
Now fix an $\eps>0$. From Lemma \ref{lem-intermed}, \eqref{n-equal}
and \eqref{2-side-eps} we see that for all $n$ large enough it holds
\begin{equation} \label{bounds}
N_\lambda(S^N_n) \, \leq  \, N_\lambda((1-\eps)\, H_n^N), \qquad
N_\lambda(S^D_n) \, \geq \, N_\lambda((1+\eps)\, H_n^D)
\end{equation}
Moreover, $f^2(x) W_\sigma(x) \to 0$ at infinity so that
\begin{equation} \label{w-eps}
(1-\eps)\, \frac{k^2\pi^2}{4 f(x)^2} \leq W_\sigma(x) +
\frac{k^2\pi^2}{4 f(x)^2} \leq (1+\eps)\, \frac{k^2\pi^2}{4
f(x)^2}\qquad \forall\, k\geq 1
\end{equation}
for all $x$ large enough uniformly in $k$. Now observe that the
sequence $\left\{k^2\pi^2/4 f(x)^2\right\}_{k\geq 1}$ enlists {\it
all} the eigenvalues of the operator $-f(x)^{-2}\, \frac{d^2}{dt^2}$
on the interval $(-1, 1)$ with Dirichlet boundary conditions. Hence
it follows from \eqref{ort-sum} and \eqref{w-eps} that for $n$ large
enough
\begin{align}
N_\lambda((1+\eps)\, H_n^D) & \geq \, N_\lambda((1+\eps)^2\, \B_n^D)+
N_\lambda((1+\eps)\, \mathcal{H}_n^D) \nonumber \\
N_\lambda((1-\eps)\, H_n^N) & \leq \, N_\lambda((1-\eps)^2\, \B_n^N)+
N_\lambda((1-\eps)\, \mathcal{H}_n^N) , \label{up}
\end{align}
where $\B_n^{N/D}$ are the operators defined in Section
\ref{auxiliary}. Note that $\B_n^{N/D}$ and $\mathcal{H}_n^{N/D}$
satisfy assumption \eqref{eps-B} by Lemma \ref{B-op} and equation
\eqref{asymp-equal}. In view of \eqref{two-sided} we then conclude
that for $n$ large enough
\begin{align}
N_\lambda(A_\sigma) & \leq \, (1+\mathcal{O}(\eps))\,
\left(N_\lambda(T_n^N)+N_\lambda(\B_n^N)
+N_\lambda(\mathcal{H}_n^N)\right) \label{almost-}\\
N_\lambda(A_\sigma) & \geq \, (1+\mathcal{O}(\eps))\,
\left(N_\lambda(T_n^D)+N_\lambda(\B_n^D)
+N_\lambda(\mathcal{H}_n^D)\right), \label{almost+}
\end{align}
If the volume of $\Omega$ is finite then it follows from Lemmas
\ref{B-op}, \ref{finite-part} and equations \eqref{asymp-equal},
\eqref{almost-}, \eqref{almost+} that for any $\eps>0$
\begin{equation*} \label{finite-vol}
1+\mathcal{O}(\eps) \, \leq\, \liminf_{\lambda\to\infty}\,
\frac{N_\lambda(A_\sigma)}{\frac{\lambda}{4\pi}\, |\Omega|
+N_\lambda(\mathcal{H}_\sigma)}\, \leq \,
\limsup_{\lambda\to\infty}\,
\frac{N_\lambda(A_\sigma)}{\frac{\lambda}{4\pi}\, |\Omega|
+N_\lambda(\mathcal{H}_\sigma)}\, \leq 1+\mathcal{O}(\eps).
\end{equation*}
By letting $\eps\to 0$ we arrive at \eqref{principle-eq}.

\vspace{0.15cm}

\noindent Case $|\Omega|=\infty$. If the volume of $\Omega$ is
infinite, then Lemma \ref{B-op} gives
\begin{equation} \label{aux-inf}
N_\lambda(T_n^N)+N_\lambda(\B_n^N) \, \sim \, N_\lambda(\B_n^N) \,
\sim\, N_\lambda(\B_n^D)\, \sim\,
N_\lambda(T_n^D)+N_\lambda(\B_n^D)\, \sim\, N_\lambda(\B)
\end{equation}
as $ \lambda\to\infty$. Moreover, mimicking all the above estimates
for the Dirichlet-Laplacian $-\Delta_\Omega^D$ instead of $A_\sigma$
it is straightforward to verify that for any $\eps>0$ and $n$ large
enough, depending on $\eps$, it holds
$$
N_\lambda((1-\eps)\, \B_n^N)\, \leq \, N_\lambda(-\Delta_\Omega^D)\, \leq \,
N_\lambda((1+\eps)\, \B_n^D).
$$
This together with \eqref{eps-B} and \eqref{aux-inf} implies that
$N_\lambda(\B) \, \sim \, N_\lambda(-\Delta_\Omega^D)$ as $ \lambda\to\infty
$. Equation \eqref{principle-eq} thus follows again from
\eqref{almost-} and \eqref{almost+}.
\end{proof}

\begin{proof}[Proof of Corollary \ref{robin-w}]
Note that the assumption \ref{ass-h} is fulfilled. Indeed, equation
\eqref{taylor} shows that $f(x) V(x)\to 0$. Consequently \eqref{compact-h}
holds true since $f\to 0$ and
\begin{equation} \label{pot-asymp}
W_\sigma(x) \, \sim\, \frac{\sigma}{f(x)}\qquad x\to \infty.
\end{equation}
To prove \eqref{robin-weyl} we recall that
$$
\liminf_{ \lambda\to\infty}\,  \lambda^{-1} N_\lambda(A_\sigma) \geq
\liminf_{ \lambda\to\infty}\,
 \lambda^{-1} N_\lambda(-\Delta_\Omega^D) = \frac{|\Omega|}{4\pi}\, .
$$
On the other hand, if $\limsup_{x\to\infty} x^2 f(x) =0$, then
\eqref{pot-asymp} says for any $\eps>0$ exists an $x_\eps$ such that
$W_\sigma(x) \geq \frac{x^2}{\eps^2}$ holds for all $x\geq x_\eps$.
Lemma \ref{enough} gives
$$
\limsup_{ \lambda\to\infty}\,  \lambda^{-1}
N_\lambda(\mathcal{H}_\sigma) \leq \limsup_{ \lambda\to\infty}\,
\lambda^{-1} N_\lambda\left( -\frac{d^2}{dx^2}\,
+\frac{x^2}{\eps^2}\right)_{L^2(x_\eps,\infty)} = \frac{\eps}{4}\, .
$$
From Lemma \ref{B-op} and the proof of Theorem \ref{principle}, see
equations \eqref{two-sided}, \eqref{finite-vol}, \eqref{bounds} and
\eqref{up} we then get
$$
\limsup_{ \lambda\to\infty}\, \frac{N_\lambda(A_\sigma)}{\lambda}
\leq \, (1+\mathcal{O}(\eps))\, \frac{|\Omega|}{4\pi} + \limsup_{
\lambda\to\infty}\, \frac{N_\lambda((1-\eps)\, \mathcal{H}_\sigma)}{
\lambda} \leq (1+\mathcal{O}(\eps))\, \frac{|\Omega|}{4\pi}
+\mathcal{O}(\eps)\, .
$$
Equation \eqref{robin-weyl} now follows by letting $\eps\to 0$. In
order to prove \eqref{robin-linear} we note that $W_\sigma(x) \sim
\sigma\, a^{-2}\, x^2$ as $x\to\infty$, see \eqref{pot-asymp}. From
Lemma \ref{enough} we thus deduce that
$$
\lim_{ \lambda\to\infty}\,  \lambda^{-1}
N_\lambda(\mathcal{H}_\sigma) \, = \, \frac{|a|}{4\sqrt{\sigma}},
$$
so that \eqref{eps} is satisfied and \eqref{robin-linear} follows
from Theorem \ref{main-thm}.
\end{proof}


\section{Generalisations}
\label{gen}

\subsection{Non symmetric boundary conditions}
\label{non-sym} As mentioned in Remark \ref{general}, the above
approach can be applied also to Robin Laplacians with different
boundary conditions on the upper and lower boundary of $\Omega$.
More precisely, to operators $A_{\sigma_1,\sigma_2}$ generated by
the closure of the form
\begin{equation} \label{form-gen}
Q_{\sigma_1,\sigma_2}[u]= \int_\Omega |\nabla u|^2\, dxdy +
\int_1^\infty\!  \left(\sigma_1(x)\, u(x,f(x))^2+\sigma_2(x)\,
u(x,-f(x))^2 \right)\, dx
\end{equation}
on $C_0^2(\overline{\Omega})$. We can proceed in the same way as in
section \ref{proof} replacing the function $v(x,y)$ in step 2 by the
function $w(x,y)$, which solves the eigenvalue problem
\begin{align}
\label{new-bc}
-\partial_y^2\, w(x,y) & = \bar\mu(x)\, w(x,y), \\
\partial_y w(x,-f(x)) = \sigma_1(x)\, w(x,-f(x)),\ \  &
\ \ \partial_y w(x,f(x)) = -\sigma_2(x)\, w(x,f(x)) \nonumber,
\end{align}
with $\bar\mu(x)$ being the principle eigenvalue. From equation
\eqref{new-limit}, see appendix, we then get a generalisation of
Theorem \ref{principle}.

\begin{proposition}
Suppose that assumptions \ref{ass-f1},\, \ref{ass-h} and
\ref{number} for $\sigma_1, \, \sigma_2$ are satisfied. Then
\begin{equation} \label{principle-gen}
N_\lambda(A_{\sigma_1,\sigma_2}) \, \sim \,
N_\lambda(-\Delta_\Omega^D) + N_\lambda(\mathcal{H}_{\bar\sigma})
\qquad \lambda\to\infty, \qquad  \bar\sigma(x) =
\frac{\sigma_1(x)+\sigma_2(x)}{2} .
\end{equation}
\end{proposition}

\subsection{Dirichlet-Neumann Laplacian}
\label{d-n} Our second remark concerns the case in which we impose
Dirichlet boundary condition on one of the boundaries of $\Omega$.
We confine ourselves to the special situation when we have Dirichlet
boundary condition on one boundary and Neumann on the other.
We denote the resulting operator by $A_{0,\infty}$.

\begin{proposition} \label{d_n}
Let $|\Omega| < \infty$ and assume that $f$ is decreasing at
infinity. Then
\begin{equation} \label{d-n-eq}
N_\lambda(A_{0,\infty}) \, \sim \, \frac{|\Omega|}{4\pi}\, \,
\lambda,  \qquad \lambda\to\infty.
\end{equation}
\end{proposition}

\begin{proof}
First we observe that by the variational principle.
\begin{equation}  \label{lowerb-d}
\liminf_{ \lambda\to\infty}\,  \lambda^{-1} N_\lambda(A_{0,\infty}) \geq
\liminf_{ \lambda\to\infty}\,
 \lambda^{-1} N_\lambda(-\Delta_\Omega^D) = \frac{|\Omega|}{4\pi}\, .
\end{equation}
Assume that  $f$ is decreasing on $(a,\infty)$ and that $\lambda$ is large enough so that there exists a unique point $x_\lambda>a$ such that $f(x_\lambda) = \pi/(4 \sqrt{\lambda})$. We impose additional Neumann boundary condition at $\{x= x_\lambda\}$ dividing thus  $\Omega$ into the finite part $\Omega_\lambda: =\{(x,y)\in\Omega\, :\, x<x_\lambda\}$ and its complement $\Omega_\lambda^c$. It is then easy to see that the quadratic form of the corresponding operator acting on $\Omega_\lambda^c$ is bounded from below by
$$
 \int_{x_\lambda}^\infty \int_{-f(x)}^{f(x)}\, \Big (\frac{\pi^2}{16\, f^2(x)}\, \, u^2 +|\partial_x u|^2\Big )\, dy\, dx\,  \geq \,
\lambda\, \int_{x_\lambda}^\infty \int_{-f(x)}^{f(x)}\, u^2\, dy\, dx
$$
for all functions $u$ from its domain. Consequently, this operator does not have any eigenvalues below $\lambda$. To estimate the number of eigenvalues of the operator acting on $\Omega_\lambda$, we cover $\Omega_\lambda$ with a finite collection of disjoint cubes of size $L = 1/(\eps\sqrt{\lambda})$ with $\eps>0$. Since $\Omega_\lambda$ has the extension property,  the standard technique of Neumann bracketing gives
\begin{align}
 \lambda^{-1} N_\lambda(A_{0,\infty})
& \leq \,  \lambda^{-1} N_\lambda(-\Delta_{\Omega_\lambda}^N)
\leq \, \frac{|\Omega_\lambda|}{4\pi}\, (1+\mathcal{O}(\eps)) +
c\, \, \frac{|\partial\Omega_\lambda|}{\sqrt{\lambda}} \, (1+\eps^{-1}) \nonumber \\
& \leq
\,  \lambda^{-1} N_\lambda(-\Delta_{\Omega_\lambda}^N)
\leq \, \frac{|\Omega_\lambda|}{4\pi}\, (1+\mathcal{O}(\eps)) +
\tilde c\, \, \frac{x_\lambda}{\sqrt{\lambda}} \, (1+\eps^{-1}), \label{upperb-n}
\end{align}
where $\tilde c$ is independent of $\lambda$.
However, since $f$ is integrable and decreasing at infinity it is easily seen that $x
f(x) \to 0$ as $x\to \infty$. Hence
$$
\limsup_{ \lambda\to\infty}\, \frac{x_\lambda}{\sqrt{\lambda}} =  \frac{4}{\pi}\,  \limsup_{ \lambda\to\infty}\,  x_\lambda\, f(x_\lambda)= 0.
$$
Letting first $\lambda\to\infty$ and then $\eps\to 0$ in \eqref{upperb-n} we obtain $\limsup_{ \lambda\to\infty}\,  \lambda^{-1} N_\lambda(A_{0,\infty}) \leq |\Omega|/4\pi$,
which together with \eqref{lowerb-d}  implies the statement.
\end{proof}



\appendix

\section{}
\label{impl}
\noindent
\begin{lemma} \label{mu-limit}
Let $\mu(x)$ be the function defined by the problem \eqref{robin}.
Then
\begin{equation} \label{mu-lim}
\mu(x) \, \leq \, \frac{\sigma(x)}{f(x)} \qquad \forall\, x > 1.
\end{equation}
\end{lemma}

\begin{proof}
For each fixed $x\in(1,\infty)$ we define the quadratic form
\begin{equation} \label{a-form}
a_x[u] = \int_{-f(x)}^{f(x)}\, |u'(y)|^2\, dy +
\sigma(x)\left(|u(f(x))|^2+|u(-f(x))|^2\right), \quad u\in D(a_x),
\end{equation}
where $D(a_x) = H^1(-f(x),f(x))$. The variational definition of
$\mu$ says that
$$
\mu(x) = \inf_{u\in D(a_x)}\,
\frac{a_x[u]}{\|u\|^2_{L^2(-f(x),f(x))}} \leq
\frac{a_x[1]}{\|1\|^2_{L^2(-f(x),f(x))}}\, =
\frac{\sigma(x)}{f(x)}\, .
$$
\end{proof}

\noindent In the next Lemma we use the notation $\kappa(x) :=
\sqrt{\mu(x)}$.

\begin{lemma}  \label{implicit}
Let the assumption \ref{ass-h} be satisfied. Then the
eigenfunction $v(x,y)$ of the problem \eqref{robin} associated to the
eigenvalue $\mu(x)$ is twice continuously differentiable in $x$.
Moreover we have
\begin{align}
\lim_{x\to\infty}\, \frac{f(x)\, \mu(x)}{\sigma(x)} & = 1 \label{mu} \\
\lim_{x\to\infty} v(x,y) & = 1 \, \, \, \qquad \text{uniformly\, in
\,} y , \label{v}\\
\lim_{x\to\infty} \partial_x v(x,y) & = 0 \, \, \, \qquad
\text{uniformly\, in\, } y \label{v'}.
\end{align}
\end{lemma}

\begin{proof} It is easy to see that
\begin{equation} \label{eq-v}
v(x,y) = \cos(\kappa(x) y),
\end{equation}
where $\kappa(x)$ is the first positive solution to the implicit
equation
\begin{equation} \label{kappa}
F(x,\kappa): =\kappa\, \tan(\kappa f(x))-\sigma(x) =0.
\end{equation}
Since $f(x) \kappa(x)\to 0$ as $x\to\infty$ by Lemma \ref{mu-lim}
(recalling that $\sigma(x)f(x)\to 0$), we easily deduce from
\eqref{kappa} that
\begin{equation} \label{g}
\lim_{x\to\infty} \, \frac{f(x) \kappa^2(x)}{\sigma(x)} = 1,
\end{equation}
which proves \eqref{mu}. Equation \eqref{v} thus follows directly
from \eqref{eq-v} and the fact that $f(x)\kappa(x)\to 0$. Next we
note that \eqref{kappa} implies
$$
0 \, < \kappa(x) \, < \frac{\pi}{2\, f(x)} \qquad \forall\, x >1,
$$
and hence
\begin{equation} \label{Fk}
\pd_\kappa F(x,\kappa) = \tan(f(x)\kappa) +\frac{f(x)\,
\kappa}{\cos^2(f(x)\kappa)}
>0.
\end{equation}
Since $\sigma\in C^2(1,\infty)$, the implicit function theorem shows
that $\kappa$ is of the class $C^2$ and in view of \eqref{eq-v} we
see that $v$ is twice continuously differentiable in $x$.

In order to prove \eqref{v'} we need some information about the
behaviour of $\kappa'$ for large $x$. From the positivity of $f$ and
$\sigma$ and from the Taylor theorem we conclude that
$\sigma'/\sqrt{\sigma}$ is bounded and that $f'/\sqrt{f}\to 0$, see
equation \eqref{taylor}. Equations \eqref{Fk} and \eqref{g} then
give
\begin{equation} \label{kappa'}
\kappa'(x) = -\frac{\pd_x F}{\pd_\kappa F}\, \sim \,
\frac{\sqrt{\sigma(x)}}{2\sqrt{f(x)}}\, \, (f'(x)-\sigma'(x)) \quad
x\to\infty.
\end{equation}
On the other hand, a direct calculation shows that
$$
|\pd_x v(x,y)|  \, \leq \, |\kappa'(x)|\, f^{3/2}(x)\,
\sqrt{\sigma(x)}\, \qquad \forall\, x>1.
$$
This implies \eqref{v'}.
\end{proof}

\noindent Notice that if we replace the eigenvalue problem
\eqref{robin} by \eqref{new-bc}, then a straightforward analysis of
the associated implicit equation shows that
\begin{equation} \label{new-limit}
\lim_{x\to\infty}\, \frac{f(x)\, \bar\mu(x)}{\bar\sigma(x)} = 1,
\qquad  \bar\sigma(x) = \frac{\sigma_1(x)+\sigma_2(x)}{2}\, .
\end{equation}

\vspace{0.2cm}

\section*{Acknowledgements}
I would like to thank the referee for very useful suggestions and remarks which helped me improve the original version of the paper. The research has been partially supported by the German Research
Foundation (DFG) under Grant KO 3636/1-1.



\vspace{0.1cm}

\begin{thebibliography}{GGMT}


\bibitem[Be]{be} M.~ van den Berg: On the spectrum of the Dirichlet
Laplacian for horn-shaped regions in $\R^n$ with infinite volume.
{\em J.~Funct.~Anal.}\,  {\bf 58} (1984) 150--156.

\vspace{0.1cm}

\bibitem[Ber]{ber} G.~Berger: Nonclassical eigenvalue asymptotic for elliptic
operators of second order in unbounded trumpet-shaped domains with
Neumann boundary conditions, {\em Math.~Nachr.} {\bf 161} (1993)
345--360.

\vspace{0.1cm}

\bibitem[BiSo]{bs} M.~S.~Birman, M.~Z.~Solomyak: The principal term of spectral asymptotics for
non-smooth elliptic problems, {\em Func. Anal. Appl.}  {\bf 4} (1971).

\vspace{0.1cm}

\bibitem[Da]{da} E.B.~Davies: {\em Heat Kernels and Spectral
Theory}, Cambridge Universtity Press (1989).

\vspace{0.1cm}

\bibitem[DS]{ds} E.B.~Davies and B.~Simon: Spectral Properties of
Neumann Laplacians of Horns. {\em Geom.~and Funct.~Analysis} {\bf 2}
(1992) 105--117.

\vspace{0.1cm}

\bibitem[EH]{eh} W.D.~Evans, D.J.~Harris: Sobolev Embeddings for
generalised ridge domains. {\em Proc.~London Math.~Soc.}\, {\bf 54}
(1987) 141--175.


\vspace{0.1cm}

\bibitem[FSW]{fsw} R.L.~Frank, B.~Simon and T.~Weidl: Eigenvalue bounds for
perturbations of Schr\"odinger operators and Jacobi matrices with
regular ground states. {\em Comm.~Math.~Phys.}\, {\bf 282} (2008)
199--208.


\vspace{0.1cm}

\bibitem[FLS]{fls} R.L.~Frank, E.H.~Lieb and R.~Seiringer:
Equivalence of Sobolev inequalities and Lieb-Thirring inequalities.
Preprint: arXiv: 0909.5449v1.

\vspace{0.1cm}

\bibitem[HSS]{hss} R.~Hempel, L.~Seco and B.~Simon:
The essential spectrum of Neumann Laplacians on some bounded
singular domains. {\em J.~Funct.~Anal.}\,  {\bf 102} (1991)
448--483.

\vspace{0.1cm}

\bibitem[Iv1]{iv1} V.~Ivrii: {\em Microlocal analysis and precise spectral
asymptotics.}. Springer Monographs in Mathematics. Springer-Verlag,
Berlin, (1998).

\vspace{0.1cm}

\bibitem[Iv2]{ivrii} V.~Ivrii: Precise spectral asymptotics for
Neumann Laplacian in domains with cusps. {\em Appl.~Anal.} \, {\bf
71} (1999) 139--147.

\vspace{0.1cm}

\bibitem[JMS]{jms} V.~Jak\v si\'{c}, S.~Mol\v canov, and B.~Simon:
Eigenvalue asymptotics of the Neumann Laplacian of regions and
manifolds with cusps. {\em J.~Funct.~Anal.}\,  {\bf 106} (1992)
59--79.


\vspace{0.1cm}

\bibitem[Ma]{mz} V.~Maz'ya: {\em Sobolev Spaces},
Springer Verlag, Berlin New York, (1985).

\vspace{0.1cm}

\bibitem[Me]{me} G. M\'{e}tivier: Valeurs propres de problemes aux limites elliptiques irregulieres,
{\em Bull. Soc. Math. France Suppl. Mem.} {bf 51-52} (1977), 125--219.

\vspace{0.1cm}

\bibitem[N]{net} Yu.~Netrusov: Sharp remainder estimates in the Weyl formula
for the Neumann Laplacian on a class of planar domains. {\em
J.~Funct.~Anal.}\, {\bf 250} (2007) 21--41.

\vspace{0.1cm}

\bibitem[NS]{ns} Yu.~Netrusov, Yu.~Safarov: Weyl asymptotic formula
for the Laplacian on Domains with Rough Boundaries. {\em
Commun.~Math.~Phys.}\, {\bf 253} (2005) 481--509.

\vspace{0.1cm}

\bibitem[RS]{rs} M. Reed and B. Simon: {\it Methods of Modern
Mathematical Physics} IV,  Academic Press, 1978.

\vspace{0.1cm}

\bibitem[Ro]{ro} V.G.~Rosenbljum: The eigenvalues of the first
boundary value problem on unbounded domains. {\em Math. USSR-Sb.}\,
{\bf 18} (1972) 235--248.

\vspace{0.1cm}

\bibitem[SV]{sv} Yu.~Safarov, D.~Vassiliev:
{\em The asymptotic distribution of eigenvalues of partial
differential operators.} Translations of Mathematical Monographs,
155. American Mathematical Society, Providence, RI, 1997.

\vspace{0.1cm}

\bibitem[Si]{si} B.~Simon: Nonclassical
Eigenvalue asymptotics. {\em J.~Funct.~Anal.}\,  {\bf 53} (1983)
84--98.

\vspace{0.1cm}

\bibitem[Sol]{sol1} M.~Z.~Solomyak:  On the discrete spectrum of a
class of problems involving the Neumann Laplacian in unbounded
domains. {\em Amer. Math. Soc. Transl. Ser. 2} {\bf  184} (1998)
233--251.



\vspace{0.1cm}

\bibitem[Ti]{ti} E.C.~Titchmarsh: {\em Eigenfunction Expansion
Associated with Second Order Differential Equations}, Vols. I, II,
Oxford Univ. Press, London/New York, 1958.

\vspace{0.1cm}

\bibitem[We]{we} H.~Weyl: Das asymptotische Verteilungsgesetz der
Eigenwerte linearer partieller Differentialgleichungen. {\em
Math.~Ann.} {\bf 71} (1912) 441--479.

\end{thebibliography}
\end{document}